\documentclass{ws-ijnt}
\newtheorem{thm}{Theorem}
\def\draftnote{SUMS OF DIGITS IN $q$-ARY EXPANSIONS}

\begin{document}

\markboth{J.C. Saunders}
{Sums of Digits in $q$-ary expansions}

%
\catchline{}{}{}{}{}
%

\title{SUMS OF DIGITS IN $q$-ARY EXPANSIONS}

\author{J.C. SAUNDERS}

\address{Department of Pure Mathematics, University of Waterloo, 200 University Avenue West\\
Waterloo, Ontario N2L 3G1, Canada\\
\email{saunders\_jc@hotmail.com} }

\maketitle

\begin{history}
\received{(26 August 2013)}
\accepted{(23 June 2014)}
\end{history}

\begin{abstract}
Let $s_q(n)$ denote the sum of the digits of a number $n$ expressed in base $q$. We study here the ratio $\frac{s_q(n^\alpha)}{s_q(n)}$ for various values of $q$ and $\alpha$. In 1978, Kenneth B. Stolarsky showed that $\lim\inf_{n\rightarrow\infty}\frac{s_2(n^2)}{s_2(n)}=0$ and that $\lim\sup_{n\rightarrow\infty}\frac{s_2(n^2)}{s_2(n)}=\infty$ using an explicit construction. We show that for $\alpha=2$ and $q\geq 2$, the above ratio can in fact be any positive rational number. We also study what happens when $\alpha$ is a rational number that is not an integer, terminating the resulting expression by using the floor function.
\end{abstract}



\section{Introduction}
We let $s_q(n)$ denote the sum of the digits of a number $n$ expressed in base $q$. For example, $s_2(n)$ denotes the number of $1$s in the binary expansion of $n$. For every $u\in\mathbb{N}$, we let $(u)_q$ denote the the expression of $u$ in base $q$ as follows. For $(u)_q$, we will use $q_1^{(a_1)}q_2^{(a_2)}...q_k^{(a_k)}=\underbrace{q_1q_1...q_1}_{a_1}\underbrace{q_2q_2...q_2}_{a_2}...\underbrace{q_kq_k...q_k}_{a_k}$ where we first have $a_1$ $q_1$s followed by $a_2$ $q_2$s, and so on up until $a_k$ $q_k$s. For example, in base $3$, $(u)_3=1^{(2)}2^{(1)}0^{(4)}2^{(2)}=112000022$ and represents the number $u=3^8+3^7+2\cdot 3^6+2\cdot 3+2\cdot 1=6561+2187+2\cdot 729+6+2=10214$. Further $s_3(u)=1+1+2+2+2=8$.
\newline
\newline
The study of the distribution of $s_q(n)$ has had a number of significant results. One of these results comes from Stolarsky who proved an asymptotic relationship for $s_2(n^2)/s_2(n)$ \cite{stolarsky}. The result gives the bounds
\begin{equation}\label{1}
\{\lfloor\log n\rfloor\}^{-1}\leq\frac{s_2(n^2)}{s_2(n)}\leq\frac{(4\log\log n)^2}{\log n}.
\end{equation}
where the left inequality always holds, and the right inequality holds for an infinite number of $n$ and
\begin{equation}\label{11}
c(h)(\log n)^{1-1/h}<\frac{s_2(n^h)}{s_2(n)}\leq 2(h\log n)^{1-1/h}
\end{equation}
where the right inequality always holds, and the left inequality holds for an infinite number of $n$ where $c(h)$ depends only on $h$. In particular this shows that the ratio $\frac{s_2(n^2)}{s_2(n)}$ can get both arbitrarily large and arbitrarily small. Lindstrom \cite{lindstrom} proved that for any polynomial $p(x)$ with integer coefficients, we have that
\begin{equation*}
\lim\sup_{n\rightarrow\infty}\frac{s_2(p(n))}{\log_2(n)}=h
\end{equation*}
where $\deg p=h$. Melfi \cite{melfi} studied the sequence of natural numbers $n$ for which the sum of the binary digits of $n$ is equal to the sum of the binary digits of $n^2$. He proved that
\begin{equation*}
\#\{n\leq N|s_2(n)=s_2(n^2)\}\gg N^{0.025}.
\end{equation*}
This lower bound was improved to $N^{1/19}$ by Hare, Laishram, and Stoll \cite{hare2}. Other work was done by Drmota and Rivat \cite{drmota} in studying the binary digits of squares. They showed that for any $k$ if one takes the set of squares $n^2$ such that $n<2^k$, we have that the two sets of numbers $\#\{n<2^k: s_{\geq k}(n^2)=m\}$ and $\#\{n<2^k: s_{<k}(n^2)=m\}$ are asymptotically equidistributed in residue classes where $s_{|<k|}(n^2)=s_2(n^2\mod 2^k)$ represents the sum of the rightmost $k$ digits in $n^2$ and $s_{|\geq k|}(n^2)=s_2(\lfloor n^2/2^k\rfloor)$ represents the sum of the remaining digits. The study of the sum of the binary digits of a number has also been extended to include approximations of irrational numbers, for example, by Rivoal, see \cite{rivoal}. Recently Hare, Laishram, and Stoll \cite{hare} managed to generalise Stolarsky's result by proving that for any polynomial $p(x)$ with integer coefficients and $\deg p\geq 2$, we have for any $q\geq 2$
\begin{equation*}
\lim\inf_{n\rightarrow\infty}\frac{s_q(p(n))}{s_q(n)}=0
\end{equation*}
and
\begin{equation*}
\lim\sup_{n\rightarrow\infty}\frac{s_q(p(n))}{s_q(n)}=\infty.
\end{equation*}
Stolarsky's bounds \eqref{1} and \eqref{11} show that the ratio $\frac{s_2(n^2)}{s_2(n)}$ can get arbitarily large and small.
We show in Theorem \ref{thm4} that this ratio $\frac{s_2(n^2)}{s_2(n)}$ will achieve every positive rational number. This is shown in Section \ref{sec2}. We break up the problem into two cases for the range of the ratio $r$. We start with the case of $0<r<1$, which involves a complicated construction of the desired number $u$ and then go on to the case of $r\geq 1$. We then devote Section \ref{sec3} to prove an analogous result for the general base $q$ case where we see that the ratio also hits every positive rational number. 
\newline
\newline
Section \ref{sec4} is devoted to the study of rational exponents that are not integers. We show that for every positive rational value of $\alpha<\frac{1}{2}$ and every integer $q\geq 2$ that
\begin{equation*}
\frac{s_q(\lfloor n^\alpha\rfloor)}{s_q(n)}
\end{equation*}
can hit every positive rational number. We also give a couple of results that generalise Stolarsky's result of the infimum and supremum of this ratio with irrational exponents. In the last section, we discuss some open problems and future directions for this research.

\section{The Case of $n^2$ in Base $2$}\label{sec2}
To study the $n^2$ case we first give a construction that proves that the ratio $\frac{s_2(n^2)}{s_2(n)}$ can hit every rational number $0<r<1$. After this, we show how to extend this construction to get all $r>0$.

\subsection{The case $0<r<1$}
The following theorem gives the value of $s_2(u^2)$ for the particular construction of $u$ that we will use.

\begin{thm}\label{thm1}
Let $(u)_2=1^{(k)}01^{(k+1)}01^{(k+2)}...01^{(k+m)}01^{(n)}$. If $k>2m(m+1)-1$ and $n>(m+1)k+d$ for some explicitly computable constant $d>0$ depending only on $m$, then $s_2(u^2)=n-mk+e_1$ and $s_2(u)=n+k(m+1)+e_2$ for some explicitly computable constants $e_1$ and $e_2$ depending only on $m$.
\end{thm}
\begin{proof}
If all the $0$s in the above binary expansion of $u$ are replaced by $1$s, then we get the expansion for $2^{c_1}-1$ where $c_1=k(m+1)+(1+2+...+m)+(m+1)+n$. The leftmost zero subtracts off $2^{c_2}$ where $c_2=km+n+(1+2+...+m)+m$. Going through the $0$s from left to right, we find that the $j$th zero subtracts off $2^{c_{j+1}}$ where $c_{j+1}=k(m-j+1)+(j+(j+1)+...+m)+(m+1-j)+n$ and $c_{m+2}=n$. Thus we have
\begin{align*}
u&=2^{c_1}-2^{c_2}-2^{c_3}-...-2^{c_{m+2}}-1\\
&=2^{c_1}-1-\sum_{i=2}^{m+2}2^{c_i}.
\end{align*}
We note that
\begin{align}
c_i-c_{i+1}=&k(m+2-i)+n+((i-1)+i+(i+1)+...+m)+(m+2-i)\nonumber\\
&-k(m+1-i)-n-(i+(i+1)+...+m)-(m+1-j)\\
=&k+i\label{3}.
\end{align}
for all $1\leq i\leq m+1$. Thus we have the following:
\begin{align*}
u^2&=\left(2^{c_1}-1-\sum_{i=2}^{m+2}2^{c_i}\right)^2\\
&=(2^{c_1}-1)^2-2(2^{c_1}-1)\left(\sum_{i=2}^{m+2}2^{c_i}\right)+\left(\sum_{i=2}^{m+2}2^{c_i}\right)^2\\
&=2^{2c_1}-2^{c_1+1}+1-\sum_{i=2}^{m+2}2^{c_i+c_1+1}+\sum_{i=2}^{m+2}2^{c_i+1}+\sum_{i=2}^{m+2}\sum_{j=2}^{m+2}2^{c_i+c_j}\\
&=2^{2c_1}-2^{c_1+1}+1-\sum_{i=2}^{m+2}2^{c_i+c_1+1}+\sum_{i=2}^{m+2}2^{c_i+1}+\sum_{i=4}^{2m+4}\sum_{j=\max\{i-(m+2),2\}}^{\min\{i-2,m+2\}}2^{c_{i-j}+c_j}.
\end{align*}
We split the last summation up so that the conditions on the indices disappear:
\begin{align*}
&2^{2c_1}-2^{c_1+1}+1-\sum_{i=3}^{m+3}2^{c_{i-1}+c_1+1}+\sum_{i=2}^{m+2}2^{c_i+1}\\
&+\sum_{i=4}^{m+3}\sum_{j=2}^{i-2}2^{c_{i-j}+c_j}+\sum_{i=m+4}^{2m+4}\sum_{j=i-(m+2)}^{m+2}2^{c_{i-j}+c_j}\\
=&2^{2c_1}-2^{c_1+1}+1-2^{c_2+c_1+1}-\sum_{i=4}^{m+3}2^{c_{i-1}+c_1+1}+\sum_{i=2}^{m+2}2^{c_i+1}\\
&+\sum_{i=4}^{m+3}\sum_{j=2}^{i-2}2^{c_{i-j}+c_j}+\sum_{i=m+4}^{2m+4}\sum_{j=i-(m+2)}^{m+2}2^{c_{i-j}+c_j}.
\end{align*}
We combine the first summation with the $j=2$ term in the third summation and pull out the $i=2m+4$ term in the last summation, which when simplified gives $2^{2n}$ since $2c_{m+2}=2n$. We also combine the first and fourth terms since $c_1-c_2-1=k$. This simplifies to:
\begin{align*}
&2^{c_2+c_1+1}(2^{k}-1)-2^{c_1+1}+1+\sum_{i=4}^{m+3}2^{c_{i-1}+c_1+1}(2^{c_{i-2}+c_2-c_{i-1}-c_1-1}-1)\\
&+\sum_{i=2}^{m+2}2^{c_i+1}+\sum_{i=4}^{m+3}\sum_{j=3}^{i-2}2^{c_{i-j}+c_j}+\sum_{i=m+4}^{2m+3}\sum_{j=i-(m+2)}^{m+2}2^{c_{i-j}+c_j}+2^{2n}.
\end{align*}
We combine the first and third summations:
\begin{align*}
&2^{c_2+c_1+1}(2^{k}-1)+2^{c_1+1}(2^{2n-c_1-1}-1)+1\\
&+\sum_{i=4}^{m+3}\left(\left(\sum_{j=3}^{i-2}2^{c_{i-j}+c_j}\right)+2^{c_{i-1}+c_1+1}(2^{c_{i-2}+c_2-c_{i-1}-c_1-1}-1)\right)\\
&+\sum_{i=2}^{m+2}2^{c_i+1}+\sum_{i=m+4}^{2m+3}\sum_{j=i-(m+2)}^{m+2}2^{c_{i-j}+c_j}\\
=&2^{c_2+c_1+1}(2^{k}-1)\\
&+\sum_{i=4}^{m+3}\left(\left(\sum_{j=3}^{i-2}2^{c_{i-j}+c_j}\right)+2^{c_{i-1}+c_1+1}(2^{c_{i-2}+c_2-c_{i-1}-c_1-1}-1)\right)\\
&+\sum_{i=m+4}^{2m+3}\sum_{j=i-(m+2)}^{m+2}2^{c_{i-j}+c_j}+2^{c_1+1}(2^{n-k(m+1)-(1+2+...+m)-(m+1)-1}-1)+\sum_{i=2}^{m+2}2^{c_i+1}+1
\end{align*}
using the fact that $2n-c_1-1=n-k(m+1)-(1+2+...+m)-(m+1)-1$. We now make a series of observations on the above expression so that we may calculate the sum of its digits. First off, for every $1\leq j<i\leq m+2$, we have the following:
\begin{align*}
c_{i-j}+c_j=&(k+1)(m-(i-j)+2)+((i-j-1)+(i-j)+...+m)\\
&+(k+1)(m-j+2)+((j-1)+j+...+m)\\
=&(k+1)(2m-i+4)+((i-j-1)+(i-j)+...+m)\\
&+((j-1)+j+...+m)+2n.
\end{align*}
Thus if there exists $i,i',j,j'$, not necessarily distinct, in the ranges above with $i-j\geq 2$, $i'-j'\geq 2$, and $c_{i-j}+c_j=c_{i'-j'}+c_{j'}$, we have
\begin{align*}
&(k+1)(-i)+((i-j-1)+(i-j)+...+m)+((j-1)+j+...+m)\\
=&(k+1)(-i')+((i'-j'-1)+(i'-j')+...+m)+((j'-1)+j'+...+m)
\end{align*}
So
\begin{equation*}
(k+1)|i-i'|\leq 2m(m+1).
\end{equation*}
If $i\neq i'$, then we would have $k+1\leq 2m(m+1)$, contradicting our assumptions on $k$ and $m$. So $i=i'$ and we have
\begin{align*}
&((i-j-1)+(i-j)+...+m)+((j-1)+j+...+m)\\
=&((i-j'-1)+(i-j')+...+m)+((j'-1)+j'+...+m)
\end{align*}
We may solve this to get $j+j'=i$. Also, for $i'-j\geq 2$, we have the following:
\begin{align*}
c_{i'-j}-c_{i'-2}=&(k+1)(m-(i'-j)+2)+[(i'-j-1)+(i'-j)+...+m]\\
&-(k+1)(m-(i'-2)+2)-[(i'-3)+(i'-2)+...+m]\\
=&(k+1)(j-2)+(i'-j-1)+(i'-j)+...+(i'-4)\\
\geq&(k+1)(j-2)+1+2+...+(j-2)\\
=&(k+1)m+[1+2+...+m]\\
&-(k+1)(m+2-j)-[(j-1)+j+...+m]\\
=&c_2-c_j.
\end{align*}
Since for $i'\leq i$, we have $c_{i'-2}\geq c_{i-2}$, we therefore have that $c_{i'-j}+c_j\geq c_{i-2}+c_2$. Also, for $i-j\geq 1$, we have the following:
\begin{align*}
c_{i-j+1}-c_{i-1}=&(k+1)(m-(i-j+1)+2)+[(i-j)+(i-j+1)+...+m]\\
&-(k+1)(m-(i-1)+2)-[(i-2)+(i-1)+...+m]\\
=&(k+1)(j-2)+[(i-j)+(i-j+1)+...+(i-3)]\\
<&(k+1)(j-1)+1+2+...+(j-2)\\
=&(k+1)(m+1)+[1+2+...+m]\\
&-(k+1)(m-j+2)-[(j-1)+j+...+m]\\
=&c_1-c_j.
\end{align*}
So we have $c_1-c_j>c_{i-j+1}-c_{i-1}$. Thus for $i'-j\geq 2$ with $i'>i$, we have $c_{i'-j}\leq c_{i-j+1}$ so that $c_1-c_j>c_{i'-j}-c_{i-1}$ or $c_{i-1}+c_1>c_{i'-j}+c_j$.
\newline
\newline
Finally, by~\eqref{3} we have
\begin{align*}
c_{i-2}+c_2-c_{i-1}-c_1-1&=(c_{i-2}-c_{i-1})-(c_1-c_2)-1\\
&=k+(i-2)-(k+1)-1\\
&=i-4.
\end{align*}
All of the above calculations when applied to our expression for $u^2$ allow us to conclude that
\begin{align*}
s_2(u^2)=&s_2(2^{c_2+c_1+1}(2^{k}-1))\\
&+\sum_{i=4}^{m+3}\left(\left(s_2\left(\sum_{j=3}^{i-2}2^{c_{i-j}+c_j}\right)\right)+s_2(2^{c_{i-1}+c_1+1}(2^{i-4}-1))\right)\\
&+\sum_{i=m+4}^{2m+3}s_2\left(\sum_{j=i-(m+2)}^{m+2}2^{c_{i-j}+c_j}\right)+s_2(2^{c_1+1}(2^{n-k(m+1)-(1+2+...+m)-(m+1)-1}-1))\\
&+s_2\left(\sum_{i=2}^{m+2}2^{c_i+1}+1\right).
\end{align*}
As well, the above calculations also show that the only terms in the above expression that depend on $n$ and/or $k$ are $s_2(2^{c_2+c_1+1}(2^{k}-1))$ and $s_2(2^{c_1+1}(2^{n-k(m+1)-(1+2+...+m)-(m+1)-1}-1))$. The rest of the terms will add up to a value that only depends on $m$, which we will call $e$. Let $d=(1+2+...+m)+(m+1)+1$ and, by assumption, suppose that $n>(m+1)k+d$. Then $s_2(u^2)=k+n-k(m+1)+e=n-mk+e$.
\end{proof}

\begin{thm}\label{thm2}
For all rational $0<r<1$, there exists a $u\in\mathbb{N}$ such that
\begin{equation*}
\frac{s_2(u^2)}{s_2(u)}=r.
\end{equation*}
\end{thm}

\begin{proof}
Let $r=\frac{a}{c}<1$ with $\gcd(a,c)=1$. Pick $m$ such that $\frac{1}{2(m+1)}<\frac{a}{c}$ and $\gcd(2m+1,c-a)=1$. We consider the construction for $u$ as given in Theorem \ref{thm1}. We have $s_2(u)=n+k(m+1)+e_2$ for some constant $e_2$ depending only on $m$ and $s_2(u^2)=n-mk+e_1$ for some constant $e_1$ also only depending on $m$ where $k>2m(m+1)-1$ and $n>(m+1)k+d$ with $d$ as defined in the proof of Theorem \ref{thm1}. We want to pick suitable values for $k$ and $n$ such that the construction of $u$ gives us exactly the ratio $\frac{a}{c}$. We achieve this as follows. Pick $t\in\mathbb{N}$ such that
\begin{equation*}
t>2m(d+e_1)+d+2e_1-e_2,
\end{equation*}
\begin{equation*}
t>e_1,
\end{equation*}
and
\begin{equation}\label{kvalue}
k:=\frac{t(c-a)+e_1-e_2}{2m+1}\in\mathbb{N}.
\end{equation}
where $t$ is sufficiently large to allow for $k>2m(m+1)$. Such a $t$ is possible by the Chinese Remainder Theorem as $\gcd(2m+1,c-a)=1$ and $a<c$. Our goal is to show that we can use \eqref{kvalue} for our value of $k$. For our value of $n$, we use $n:=ta+mk-e_1\in\mathbb{N}$. We see this is positive as $t>e_1$. We need to show that
\begin{equation*}
(m+1)k+d<n.
\end{equation*}
so that we can apply Theorem \ref{thm1}. To achieve this, we show that
\begin{equation*}
k+d+e_1<ta.
\end{equation*}
 We have $2m(d+e_1)+d+2e_1-e_2<t\leq2(m+1)ta-tc$. Thus we have the following:
\begin{align*}
\frac{t(c-a)+e_1-e_2}{2m+1}+d+e_1&=\frac{tc-ta+e_1-e_2+2md+d+2me_1+e_1}{2m+1}\\
&=\frac{tc-ta+2m(d+e_1)+d+2e_1-e_2}{2m+1}\\
&<\frac{tc-ta+2(m+1)ta-tc}{2m+1}\\
&=\frac{(2m+1)ta}{2m+1}=ta.
\end{align*}
Hence $(m+1)k+d<n$ as required for the hypothesis of Thoerem $1$. We thus can now use Theorem \ref{thm1}. Let $(u)_2=1^{(k)}01^{(k+1)}01^{(k+2)}...01^{(k+m)}01^{(n)}$. Thus we have
\begin{align*}
\frac{s_2(u^2)}{s_2(u)}&=\frac{n-mk+e_1}{n+k(m+1)+e_2}\\
&=\frac{(ta+mk-e_1)-mk+e_1}{(ta+mk-e_1)+k(m+1)+e_2}\\
&=\frac{ta}{ta-e_1+\left(\frac{t(c-a)+e_1-e_2}{2m+1}\right)(2m+1)+e_2}\\
&=\frac{ta}{ta-e_1+t(c-a)+e_1-e_2+e_2}\\
&=\frac{ta}{tc}=r.
\end{align*}
This proves the result as required.
\end{proof}
Having obtained our result for $0<r<1$, we now go on to the case of $r\geq 1$.

\begin{lemma}\label{lem15}
If there exists a $u\in\mathbb{N}$ such that
\begin{equation}\label{ratio}
\frac{s_2(u^2)}{s_2(u)}=r,
\end{equation}
then there exists a $v\in\mathbb{N}$ such that
\begin{equation}
\frac{s_2(v^2)}{s_2(v)}=\frac{3}{2}r.
\end{equation}
\end{lemma}

\begin{proof}
Let $u\in\mathbb{N}$ hold for such an $r$ in~\eqref{ratio}. Take $v=(2^{2w+1}+1)u$ where $2^w>u$. Then we have
\begin{equation*}
v^2=(2^{2w+1}+1)^2u^2=2^{4w+2}u^2+2^{2w+2}u^2+u^2
\end{equation*}
We have that $2^{2w}>u^2$ and $2^{4w+2}>2^{2w+2}u^2$. Thus
\begin{equation*}
s_2(v^2)=3\cdot s_2(u^2).
\end{equation*}
So we have
\begin{equation*}
\frac{s_2(v^2)}{s_2(v)}=\frac{3\cdot s_2(u^2)}{2\cdot s_2(u)}=\frac{3}{2}r.
\end{equation*}
\end{proof}

\begin{thm}\label{thm4}
For all positive rationals $r$ there exists a $u\in\mathbb{N}$ such that
\begin{equation}\label{2}
\frac{s_2(u^2)}{s_2(u)}=r.
\end{equation}
\end{thm}

\begin{proof}
Find $s\in\mathbb{N}$ such that $r_0=\left(\frac{2}{3}\right)^sr<1$. Construct $u$ such that $\frac{s_2(u^2)}{s_2(u)}=r_0$, which we can do by Theorem \ref{thm2}, and then use $s$ applications of Lemma \ref{lem15}.
\end{proof}

\section{General Base $q$ for the Case of $n^2$}\label{sec3}
Having proved that $\frac{s_2(u^2)}{s_2(u)}$ hits every positive rational number, we now go on to prove the analogous result for the general base. Again, we start with the range $0<r<1$.

\subsection{The case of $0<r<1$}
The following theorem gives the value of $s_q(u^2)$ for the particular construction of $u$ that we will use.

\begin{thm}\label{thm5}
Let $q\geq 3$ and $(u)_q=(q-1)^{(k)}(q-2)(q-1)^{(k+1)}(q-2)(q-1)^{(k+2)}...(q-2)(q-1)^{(k+m)}(q-2)(q-1)^{(n)}$ with $k>2m(m+1)-1$ and $n>(m+1)k+d$ for some constant explicitly computable constant $d>0$, depending only on $m$ and $q$. Then we have $s_q(u^2)=(q-1)(n-mk)+e_1$ and $s_q(u)=(q-1)(k(m+1)+n)+e_2$ for some explicitly computable constants $e_1$ and $e_2$ depending only on $m$ and $q$.
\end{thm}

\begin{proof}
Same as the proof of Theorem \ref{thm1} with $2$ replaced by $q$.
\end{proof}
We now prove our result for every rational $r$ with $0<r<1$. First observe
\begin{lemma}\label{lem1}
For all $u\in\mathbb{N}$ and for all $q\geq 2$, we have that $(q-1)|u$ if and only if $(q-1)|s_q(u)$.
\end{lemma}
We will use Lemma \ref{lem1} to prove the following theorem:
\begin{thm}\label{thm6}
For all rational $0<r<1$, there exists a $u\in\mathbb{N}$ such that
\begin{equation*}
\frac{s_q(u^2)}{s_q(u)}=r.
\end{equation*}
\end{thm}
\begin{proof}
Let $r=\frac{a}{c}<1$. Pick $m\in\mathbb N$ such that $(q-1)|(m+1)$, $\frac{1}{2(m+1)}<r<1$, and $\gcd(c-a,2m+1)=1$. For example, you could pick a value of $m$ sufficiently large with $2m+1$ being prime and
\begin{equation*}
2m+1\equiv -1\mod(2q-2),
\end{equation*}
which is possible by Dirichlet's Theorem. We can pick values $d$, $e_1$, $e_2$, and $k>2m(m+1)$ such that for every construction of $(u)_q$ in Theorem \ref{thm5}, we have $s_q(u)=(q-1)(k(m+1)+n)+e_2$ and $s_q(u^2)=(q-1)(n-mk)+e_1$ with $d$ as defined in the proof of Theorem \ref{thm1}. We can see that $e_2=(q-2)(m+1)+(q-1)t_m$ where $t_m=1+2+...+m$. Since $(q-1)|(m+1)$, we have $(q-1)|e_2$ and so $(q-1)|s_q(u)$. By Lemma \ref{lem1}, we thus have $(q-1)|u$ and so $(q-1)|u^2$. By Lemma \ref{lem1} again, we have $(q-1)|s_q(u^2)$ and so $(q-1)|e_1$. Let $e_3=e_2/(q-1)$ and $e_4=e_1/(q-1)$ so that $s_q(u)=(q-1)(k(m+1)+n+e_3)$ and $s_q(u^2)=(q-1)(n-mk+e_4)$.
\newline
\newline
The rest of the proof is the same as the proof of Theorem \ref{thm2} with $e_1$ replaced by $e_4$ and $e_2$ replaced by $e_3$.
\end{proof}

\subsection{The case of $r\geq 1$}
Our proof for the case of $r\geq 1$ in the general base $q$ case uses a construction almost identical with the construction used for the case of $0<r<1$. Unlike the base $2$ case, however, our construction for the proof of the ratio hitting every rational number $r$ in the range $0<r<1$ cannot be extended to all positive rational ratios. Thus, we require a slightly different construction to prove that we can hit every rational ratio $r$ in the range $\frac{1}{2}<r<1$. Indeed, we require three different constructions, one for $q=3$, another for $q=4$, and a last one for the most general case $q\geq 5$. We then extend our constructions to cover all greater positive rational ratios like in the base $2$ case. We first focus on the last case $q\geq 5$.

\begin{lemma}\label{lem2}
Let $(u)_q=(q-1)^{(k)}0(q-1)^{(n)}$ with $q\geq 5$, $n\geq k+2$, and $k\geq 1$. Then $s_q(u^2)=(q-1)(n+1)$ and $s_q(u)=(q-1)(n+k)$.
\end{lemma}
\begin{proof}
We have $u=q^{k+n+1}-(q-1)q^{n}-1=q^{k+n+1}-q^{n+1}+q^n-1$. Thus we have
\begin{align}
u^2&=(q^{k+n+1}-q^{n+1}+q^{n}-1)^2\\
&=q^{2k+2n+2}+q^{2n+2}+q^{2n}+1-2q^{k+2n+2}+2q^{k+2n+1}-2q^{k+n+1}-2q^{2n+1}+2q^{n+1}-2q^n\\
&=q^{2k+2n+2}-2q^{k+2n+2}+2q^{k+2n+1}+q^{2n+2}-2q^{2n+1}+q^{2n}-2q^{k+n+1}+2q^{n+1}-2q^n+1\\
&=q^{k+2n+2}(q^{k}-2)+2q^{k+2n+1}+q^{2n+1}(q-2)+q^{k+n+1}(q^{n-k-1}-2)+(2q-2)q^n+1\label{4}
\end{align}
Thus, we have
\begin{align*}
s_q(u^2)&=k(q-1)-1+2+(q-2)+(q-1)(n-k-1)-1+q-1+1\\
&=(q-1)(n-1)+2q-2\\
&=(q-1)(n+1).
\end{align*}
\end{proof}

\begin{lemma}\label{lem3}
For every $q\geq 5$ and rational number $r$ with $1/2<r<1$, there exists $u\in\mathbb{N}$ such that
\begin{equation*}
\frac{s_q(u^2)}{s_q(u)}=r.
\end{equation*}
\end{lemma}

\begin{proof}
Write $r=\frac{4a}{4c}$ where $\gcd(a,c)=1$ and $\frac{1}{2}<r<1$. Let $n:=4a-1$ and $k:=4c-4a+1$. Notice that $n\geq k+2$ because $4c<8a$ and so $4c+3<8a$ so that $(4c-4a+1)+2\leq 4a-1$. Also, $a<c$ so $1<4c-4a+1=k$. By Lemma \ref{lem2}, we have for $(u)_q=(q-1)^{(k)}0(q-1)^{(n)}$,
\begin{equation*}
\frac{s_q(u^2)}{s_q(u)}=\frac{(q-1)4a}{(q-1)(4a-1+4c-4a+1)}=\frac{4a}{4c}=r.
\end{equation*}
\end{proof}
Now we calculate out $s_q(2u^2)$:

\begin{lemma}\label{lem4}
Let $u$ be as in Lemma \ref{lem2} with $q\geq 5$. We have $s_q(2u^2)=(q-1)(n+1)$.
\end{lemma}

\begin{proof}
We have from~\eqref{4} that
\begin{equation*}
2u^2=q^{k+2n+2}(2\cdot q^k-4)+4q^{k+2n+1}+q^{2n+1}(2q-4)+q^{k+n+1}(2q^{n-k-1}-4)+(4q-4)q^n+2.
\end{equation*}
Thus
\begin{align*}
s_q(2u^2)&=1+k(q-1)-3+4+1+(q-4)+1+(n-k-1)(q-1)-3+3+(q-4)+2\\
&=(n-1)(q-1)+2q-2\\
&=(q-1)(n+1).
\end{align*}
\end{proof}

\begin{lemma}\label{lem5}
For every $q\geq 5$ and rational number $r$ with $1/2<r$, there exists $v\in\mathbb{N}$ such that
\begin{equation*}
\frac{s_q(v^2)}{s_q(v)}=r.
\end{equation*}
\end{lemma}

\begin{proof}
Let $t_d=1+2+...+d$. Take $1/2<r$. Pick $d\in\mathbb{N}$ such that $\frac{d+1}{2}<2r<d+1$, for example $d=\lfloor2r\rfloor$. We have that
\begin{equation*}
\frac{1}{2}<\frac{2r}{d+1}<1.
\end{equation*}
Since $t_d=\frac{d(d+1)}{2}$, we have that
\begin{equation*}
\frac{1}{2}<\frac{d}{t_d}r<1.
\end{equation*}
By Lemma \ref{lem3}, there exists $u\in\mathbb{N}$ such that $(u)_q=(q-1)^{(k)}0(q-1)^{(n)}$, $n\geq k+2$, and $k\geq 1$, and
\begin{equation*}
\frac{s_q(u^2)}{s_q(u)}=\frac{d}{t_d}r.
\end{equation*}
Take $v=(q^{(2^d)(m+1)}+q^{(2^{d-1})(m+1)}+...+q^{2(m+1)})u$ where $q^m>u$. Then we have
\begin{equation}\label{103}
v^2=\sum_{e_1=1}^{d}\sum_{e_2=1}^{d}q^{(2^{e_1}+2^{e_2})(m+1)}u^2
\end{equation}
Suppose for some $1\leq e_1,e_2,e_3,e_4\leq d$ we have:
\begin{equation*}
2^{e_1}+2^{e_2}=2^{e_3}+2^{e_4}.
\end{equation*}
Without loss of generality, assume $e_1\geq e_3$ and $e_4\geq e_2$. So we have
\begin{equation*}
2^{e_3}(2^{e_1-e_3}-1)=2^{e_2}(2^{e_4-e_2}-1).
\end{equation*}
Thus $e_3=e_2$ and $e_1=e_4$ and so any two terms in~\eqref{103} that don't have this condition are separated by a factor of at least
\begin{equation*}
q^{2(m+1)}>q^{2m+1}>2u^2.
\end{equation*}
Using our derived fact that $s_q(u^2)=s_q(2u^2)$ from Lemma \ref{lem3}, we have that
\begin{equation*}
s_q(v^2)=t_ds_q(u^2).
\end{equation*}
Also $s_q(v)=ds_q(u)$ so that
\begin{equation*}
\frac{s_q(v^2)}{s_q(v)}=\frac{t_ds_q(u^2)}{ds_q(u)}=\frac{t_d}{d}\frac{d}{t_d}r=r.
\end{equation*}
\end{proof}
We now focus on the construction for $q=4$.

\begin{lemma}\label{lem10}
Let $(u)_4=13^{(k)}23^{(n)}$ with $n\geq k+3$ and $k\geq 1$. Then $s_4(u^2)=3n$ and $s_4(u)=3+3(k+n)$.
\end{lemma}

\begin{proof}
We have $u=2\cdot 4^{n+k+1}-4^n-1$. Thus we have
\begin{align}
u^2&=4^{2n+2k+3}-4^{2n+k+2}+4^{2n}-4^{n+k+2}+2\cdot 4^n+1\\
&=4^{2n+k+2}(4^{k+1}-1)+4^{n+k+2}(4^{n-k-2}-1)+2\cdot 4^n+1\label{13}
\end{align}
Thus, we have
\begin{equation*}
s_4(u^2)=3(k+1)+3(n-k-2)+2+1=3n.
\end{equation*}
\end{proof}

\begin{lemma}\label{lem11}
For every rational number $r$ with $1/2<r<1$, there exists $u\in\mathbb{N}$ such that
\begin{equation*}
\frac{s_4(u^2)}{s_4(u)}=r.
\end{equation*}
\end{lemma}

\begin{proof}
Write $r=\frac{2a}{2c}$ where $\gcd(a,c)=1$ and $\frac{1}{2}<r<1$. Let $n:=2a>0$ and $k:=2c-2a-1>0$. Notice that $n\geq k+3$ because $2c<4a$ and so $2c+2\leq 4a$ so that $(2c-2a-1)+3\leq 2a$. By Lemma \ref{lem10}, we have for $(u)_4=13^{(k)}23^{(n)}$,
\begin{equation*}
\frac{s_4(u^2)}{s_4(u)}=\frac{6a}{3+3(2c-2a-1+2a)}=\frac{6a}{6c}=r.
\end{equation*}
\end{proof}

\begin{lemma}\label{lem12}
Let $u$ be as in Lemma \ref{lem10}. We have $s_4(2u^2)=3n$.
\end{lemma}

\begin{proof}
We have from \eqref{13} that
\begin{equation*}
2u^2=4^{2n+2k+3}+4^{2n+k+2}(4^{k+1}-2)+4^{2n}+4^{n+k+2}(4^{n-k-2}-2)+4^{n+1}+2.
\end{equation*}
Thus $s_4(2u^2)=1+3(k+1)-1+1+3(n-k-2)-1+3=3n$.
\end{proof}

\begin{lemma}\label{lem13}
For every rational number $r$ with $1/2<r$, there exists $v\in\mathbb{N}$ such that
\begin{equation*}
\frac{s_4(v^2)}{s_4(v)}=r.
\end{equation*}
\end{lemma}

\begin{proof}
Same as proof of Lemma \ref{lem5} with $q=4$.
\end{proof}
Finally, we focus on the construction for $q=3$.

\begin{lemma}\label{lem6}
Let $(u)_3=12^{(k)}12^{(n)}$ with $n\geq k+2$ and $k\geq 1$. Then $s_3(u^2)=2n+2$ and $s_3(u)=2+2(k+n)$.
\end{lemma}

\begin{proof}
We have $u=2\cdot 3^{n+k+1}-3^n-1$. Thus we have
\begin{align}
u^2&=4\cdot 3^{2n+2k+2}-4\cdot 3^{2n+k+1}+3^{2n}-4\cdot 3^{n+k+1}+2\cdot 3^n+1\\
&=3^{2n+2k+3}+3^{2n+k+1}(3^{k+1}-4)+3^{n+k+1}(3^{n-k-1}-4)+2\cdot 3^n+1\label{12}
\end{align}
Thus, we have
\begin{align*}
s_3(u^2)&=1+2(k+1)-1+2(n-k-1)-1+2+1\\
&=2n+2.
\end{align*}
\end{proof}

\begin{lemma}\label{lem7}
For every rational number $r$ with $1/2<r<1$, there exists $u\in\mathbb{N}$ such that
\begin{equation*}
\frac{s_3(u^2)}{s_3(u)}=r.
\end{equation*}
\end{lemma}

\begin{proof}
Write $r=\frac{3a}{3c}$ where $\gcd(a,c)=1$ and $\frac{1}{2}<r<1$. Let $n:=3a-1>0$ and $k:=3c-3a>0$. Notice $n\geq k+2$ because $3c<6a$ and so $3c+3\leq 6a$ so that $3(c-a)+2\leq 3a-1$. By Lemma \ref{lem6}, we have for $(u)_3=12^{(n)}12^{(k)}$,
\begin{equation*}
\frac{s_3(u^2)}{s_3(u)}=\frac{2(3a-1)+2}{2+2(3c-3a+3a-1)}=\frac{6a}{6c}=r.
\end{equation*}
\end{proof}

\begin{lemma}\label{lem8}
Let $u$ be as in Lemma \ref{lem6}. We have $s_3(2u^2)=2n+2$.
\end{lemma}

\begin{proof}
We have from \eqref{12} that
\begin{equation*}
2u^2=7\cdot 3^{2n+2k+2}+3^{2n+k+1}(3^{k+1}-8)+3^{2n}+3^{n+k+1}(3^{n-k-1}-8)+4\cdot 3^n+2.
\end{equation*}
Thus $s_3(2u^2)=3+2(k+1)-3+1+2(n-k-1)-3+4=2n+2$.
\end{proof}

\begin{lemma}\label{lem9}
For every rational number $r$ with $1/2<r$, there exists $v\in\mathbb{N}$ such that
\begin{equation*}
\frac{s_3(v^2)}{s_3(v)}=r.
\end{equation*}
\end{lemma}

\begin{proof}
Same as proof of Lemma \ref{lem5} with $q=3$.
\end{proof}

\begin{thm}\label{thm8}
For every rational number $r$ with $1/2<r$ and $q\geq 3$, there exists $v\in\mathbb{N}$ such that
\begin{equation*}
\frac{s_q(v^2)}{s_q(v)}=r.
\end{equation*}
\end{thm}

\begin{proof}
Combination of Lemmas \ref{lem5}, \ref{lem13}, and \ref{lem9}.
\end{proof}

\section{Fractional Exponents}\label{sec4}
In studying the $n^2$ case, a natural question to ask is about other powers of $n$. Here we study where $n$ is raised to a fractional exponent. Since such a number is irrational in most cases, we take the floor of the function in order to make sense of summing the digits of the resulting number. We have the following:
\begin{thm}
For every rational number $r>0$, $q\geq 2$, $h\geq 1$, and $m\geq 3$ with $\gcd(h,m)=1$ and $\frac{h}{m}<\frac{1}{2}$, there exists a $u\in\mathbb{N}$ such that
\begin{equation*}
\frac{s_q(\lfloor u^{h/m}\rfloor)}{s_q(u)}=r.
\end{equation*}
\end{thm}
\begin{proof}
Our proof involves using the binomial theorem, this time with fractional exponents. We construct a number $u$ such that we have
\begin{equation}\label{powers}
(q^n+d)^{m/h}<u<(q^n+d+1)^{m/h}
\end{equation}
where $q^n>d$ and $h/m$ is our exponent. Thus we can see that $\lfloor u^{h/m}\rfloor=q^n+d$. We study the terms in the possibly infinite binomial expansion of each of the powers of $m/h$ in \eqref{powers} to get a range of values for $u$ that with a range of values for $d$ gives us any rational number. We begin as follows.
\newline
\newline
Let $h=h_1h_2$ so that $\gcd(h_2,q)=1$ and for all $h'|h$ with $h'>h_2$, we have that $\gcd(h',q)>1$. By Fermat's Little Theorem there exists $j\in\mathbb{N}$ such that $h_2|(q^j-1)$. Pick $i\in\mathbb{N}$ such that $h|m(q^j-1)q^i$. Pick $a,c\in\mathbb{N}$ such that
\begin{equation*}
r=\frac{a}{c}.
\end{equation*}
Pick a $w\in\mathbb{N}$ such that
\begin{equation*}
1+s_q\left(\frac{m(q^j-1)q^i}{h}\right)\leq(q-1)wc
\end{equation*}
and
\begin{equation*}
j+i\leq wa.
\end{equation*}
Also, let $d=(q^j-1)q^{wa-j}-1$ and choose $l\in\mathbb{N}$ such that $d<q^l$. Let $t_1=\sum_{k=\lfloor m/h\rfloor+1}^{\infty}{m/h\choose k}q^{l\left(\frac{m}{h}-k\right)}d^k$. $t_1$ is finite since $d<q^l$. Choose $e\in\mathbb{N}$ such that $1+s_q\left(\frac{m(q^j-1)q^i}{h}\right)+s_q(e)=(q-1)wc$ and that $t_1<e$. Finally, let $t_2=\max\{{m/h\choose k}:2\leq k\leq\lfloor m/h\rfloor\}$ and choose $n\in\mathbb{N}$ so that
\begin{equation*}
n>l,
\end{equation*}
\begin{equation*}
t_2d^{\lfloor m/h\rfloor}\leq q^n,
\end{equation*}
\begin{equation*}
e<q^{n(\frac{m}{h}-2)},
\end{equation*}
\begin{equation*}
\frac{m(d+1)}{h}<q^n,
\end{equation*}
and
\begin{equation*}
\frac{nm}{h}\in\mathbb{N}.
\end{equation*}
We construct $u$ by taking into account the first two terms of the binomial expansion of $(q^n+(d+1))^{m/h}$, which with the above considerations leads to our desired ratio. Let $u=q^{\frac{nm}{h}}+\frac{m(d+1)}{h}q^{\frac{nm}{h}-n}+e$. We have the following:
\begin{align*}
s_q(u)&=1+s_q\left(\frac{m(d+1)}{h}\right)+s_q(e)\\
&=1+s_q\left(\frac{m(q^j-1)q^i}{h}\right)+s_q(e)\\
&=(q-1)wc.
\end{align*}
By the binomial theorem, we have the following:
\begin{align*}
(q^n+d+1)^{\frac{m}{h}}&>q^{\frac{nm}{h}}+\frac{m(d+1)}{h}q^{n\left(\frac{m}{h}-1\right)}+q^{n(\frac{m}{h}-2)}\\
&>q^{\frac{nm}{h}}+\frac{m(d+1)}{h}q^{n\left(\frac{m}{h}-1\right)}+e=u
\end{align*}
where we have used the fact that
\begin{equation*}
q^{n(\frac{m}{h}-2)}<\sum_{k=2}^{\infty}{m/h\choose k}q^{n(\frac{m}{h}-k)}(d+1)^k
\end{equation*}
obtained by noting that
\begin{equation*}
q^{n(\frac{m}{h}-2)}<{m/h\choose 2}q^{n(\frac{m}{h}-2)}(d+1)^2,
\end{equation*}
that
\begin{equation*}
{m/h\choose 3}q^{n(\frac{m}{h}-3)}(d+1)^3>0,
\end{equation*}
and that each successive term is less than its predecessor.
Also, by the binomial theorem, we have the following:
\begin{align*}
(q^n+d)^{\frac{m}{h}}&=q^{\frac{nm}{h}}+\frac{md}{h}q^{\frac{nm}{h}-n}+\sum_{k=2}^{\infty}{m/h\choose k}q^{n\left(\frac{m}{h}-k\right)}d^k\\
&=q^{\frac{nm}{h}}+\frac{md}{h}q^{\frac{nm}{h}-n}+\sum_{k=2}^{\lfloor m/h\rfloor}{m/h\choose k}q^{n\left(\frac{m}{h}-k\right)}d^k+\sum_{k=\lfloor m/h\rfloor+1}^{\infty}{m/h\choose k}q^{n\left(\frac{m}{h}-k\right)}d^k\\
&<q^{\frac{nm}{h}}+\frac{md}{h}q^{\frac{nm}{h}-n}+t_2\left\lfloor\frac{m}{h}\right\rfloor q^{n(\frac{m}{h})-2n}d^{\lfloor\frac{m}{h}\rfloor}+t_1\\
&\leq q^{\frac{nm}{h}}+\frac{md}{h}q^{\frac{nm}{h}-n}+\left(\frac{m}{h}\right)q^{n(\frac{m}{h})-n}+e\\
&=q^{\frac{nm}{h}}+\frac{m(d+1)}{h}q^{\frac{nm}{h}-n}+e=u.
\end{align*}
Thus
\begin{equation*}
\lfloor u^{h/m}\rfloor=q^n+d.
\end{equation*}
So we have that
\begin{align*}
s_q(\lfloor u^{h/m}\rfloor)&=s_q(q^n+d)\\
&=1+s_q(d)\\
&=1+(q-2)+(q-1)(j-1+wa-j)\\
&=(q-1)wa
\end{align*}
where we have used the expansion $(d)_q=(q-1)^{(j)}(q-2)(q-1)^{(wa-j-1)}$. Thus we have
\begin{equation*}
\frac{s_q(\lfloor u^{h/m}\rfloor)}{s_q(u)}=\frac{(q-1)wa}{(q-1)wc}=r.
\end{equation*}
\end{proof}
It is worth pointing out that the above method of proof certainly cannot work for when the exponent $m/h>1$. It also cannot work even when $m/h\geq 1/2$ since we used the fact that
\begin{equation*}
\lim_{n\rightarrow\infty}q^{n(\frac{m}{h}-2)}=\infty
\end{equation*}
to allow our choice of $e$ to be arbitarily large. The specific case for when the exponent is $1/2$, however, can be included since this case is simply the inverse of our work on $n^2$ proved earlier. We now give our two results for the infimum and supremum of this ratio when the exponent $\alpha$ is irrational.
\begin{thm}
Let $\alpha\in\mathbb{R}$ be a positive irrational number. Then we have
\begin{equation*}
\lim\sup_{n\rightarrow\infty}\frac{s_q(\lfloor n^{\alpha}\rfloor)}{s_q(n)}=\infty.
\end{equation*}
\end{thm}
\begin{proof}
We prove our result by considering $n=q^k$ for $k\in\mathbb{N}$ and showing that $s_q(\lfloor q^{k\alpha}\rfloor)$ can get arbitarily large. Let $m\in\mathbb{N}$. Choose $j\in\mathbb{N}$ such that $m<(q-1)j$ and choose $\epsilon>0$ such that $1-\frac{1}{q^j}<q^{-\epsilon}<1$. Also, choose $k\in\mathbb{N}$ such that $k\alpha>j$ and that $k\alpha>\lceil k\alpha\rceil-\epsilon$. Such a $k$ is possible because $\{\{k\alpha\}:k\in\mathbb{N}\}$ (where $\{k\alpha\}=k\alpha-\lfloor k\alpha\rfloor$) is a dense subset of $(0,1)$. Thus, we have the following:
\begin{align*}
q^{\lceil k\alpha\rceil}>q^{k\alpha}&>q^{\lceil k\alpha\rceil-\epsilon}\\
&>q^{\lceil k\alpha\rceil}\left(1-\frac{1}{q^j}\right)\\
&>q^{\lceil k\alpha\rceil}-q^{\lceil k\alpha\rceil-j}\\
&=q^{\lceil k\alpha\rceil-j}(q^j-1).
\end{align*}
Letting $u=\lfloor q^{k\alpha}\rfloor$, we thus we have the following:
\begin{align*}
s_q(u)>s_q(q^{\lceil k\alpha\rceil-j}(q^j-1))=s_q(q^j-1)=(q-1)j>m.
\end{align*}
Since $s_q(q^k)=1$ and $m\in\mathbb{N}$ was arbitrary, we have our result.
\end{proof}
\begin{thm}
Let $0<\alpha<1$ be irrational. Then we have
\begin{equation*}
\lim\inf_{n\rightarrow\infty}\frac{s_q(\lfloor n^{\alpha}\rfloor)}{s_q(n)}=0.
\end{equation*}
\end{thm}
\begin{proof}
Let $0<\alpha<1$ be irrational. We have that for all $(q^k)^{1/\alpha}<u<(q^k+1)^{1/\alpha}$, $\lfloor u^{\alpha}\rfloor=q^k$. We will show that you can pick such a $u$ such that $s_q(u)$ is arbitrarily large, assuming $k$ is large enough. Let $m\in\mathbb{N}$. Pick $1<r<1/\alpha$ with $r\in\mathbb{Q}$ and pick $k\in\mathbb{N}$ such that $(q-1)k(r-1)>m$. Suppose $0\leq v<q^{kr-k}$, $v\in\mathbb{N}$. We have the following:
\begin{align*}
q^{k/\alpha}<\lceil q^{k/\alpha}\rceil+v&<q^{k/\alpha}+q^{kr-k}\\
&<q^{k/\alpha}+q^{k(\frac{1}{\alpha}-1)}\\
&<q^{k/\alpha}+\frac{1}{\alpha}q^{k(\frac{1}{\alpha}-1)}\\
&<(q^k+1)^{1/\alpha}.
\end{align*}
Let $\lceil q^{k/\alpha}\rceil\equiv c\mod q^{kr-k}$ where $0\leq c<q^{kr-k}$. Let $v=q^{kr-k}-1-c$. Then we have $\lceil q^{k/\alpha}\rceil+v\equiv c-1-c\equiv -1\mod q^{kr-k}$. Thus $s_q(\lceil q^{k/\alpha}\rceil+v)\geq (q-1)k(r-1)>m$. Since $s_q(\lfloor(\lceil q^{k/\alpha}\rceil+v)^{\alpha}\rfloor)=s_q(q^k)=1$ and $m\in\mathbb{N}$ was arbitrary, we have our result.
\end{proof}

\section{Conclusions}
The ratio $\frac{s_q(p(n))}{s_q(n)}$ is clearly a very complex function for any given function $p(n)$. We have shown that its range is all of $\mathbb{Q}^+$ for $p(n)=n^2$ and for $p(n)=\lfloor n^\alpha\rfloor$ for all $\alpha\in\mathbb{Q}\cap(0,1/2]$. Other questions remain, however. We still have yet to show it for $p(n)=\lfloor n^\alpha\rfloor$ for all $\alpha>1/2$ and irrational values of $\alpha$, as well as the higher integer powers $n^3$, $n^4$,... The techniques employed for the $n^2$ case would involve too many terms to be worked out by our use of the binomial theorem for higher powers so a different technique might have to be used for these cases. As well, a different technique would have to be used for irrational values of $\alpha$ in $\lfloor n^{\alpha}\rfloor$. Due to the randomness of the expansion of irrationals, however, we do conjecture that the ratio will hit every rational number. We just don't see how to prove this. Another direction is how often a particular ratio is hit. In other words, can we say anything about $\#\{n\leq N|\frac{s_q(n^2)}{s_q(n)}=\frac{a}{c}\}$ for some positive rational $\frac{a}{c}$? This has been done for the special case $\frac{a}{c}=1$ in \cite{hare2} and \cite{melfi}. This also leads to the major open question posed by Stolarsky of obtaining the average value of this ratio.
\newline
\newline
We can ask similar questions of more complicated functions as well, such as polynomial functions, exponential functions, logarithmic functions, etc. and any combination of these. Again, we would take the floor function of the result of any of these functions to make sense of summing the digits of the result. Can we categorise which functions output ratios that hit every rational number? Can we categorise which functions output ratios that get arbitrarily close to every rational number and consequently to very real number? And for such functions can we give the exact range of the ratio if it doesn't hit every rational number? And even if it doesn't get arbitarily close to everything, can we give its range regardless? We can even include the study of bounded functions in this question, such as trigonometric functions. What if we included fractions that have terminating $q$-ary expansions, like terminating decimals in base $10$? Which functions would fall under one of these categories for a given base, but fall under a different category if a different base is used?

\section{Acknowledgements}
I would like to thank Dr. Kevin Hare in making the reserach for this paper possible. I would also like to thank NSERC and the Department of Pure Mathematics, University of Waterloo for helping support this research. This research was supported in part by NSERC and the Department of Pure Mathematics, University of Waterloo.

\end{document}